\documentclass[a4paper,12pt]{article}
\usepackage{geometry,latexsym,amssymb}
\usepackage[latin5]{inputenc}
\usepackage{enumerate}
\usepackage{enumitem}
\usepackage[usenames]{color}
\usepackage{graphicx}
\usepackage{enumerate}
\usepackage{rotating}
\usepackage{multirow}
\usepackage{cite}
\usepackage{amsthm}
\usepackage{caption}
\usepackage{tikz}
\usetikzlibrary{matrix,arrows}
\usepackage{amsmath,amscd}
\usepackage{calc}
\usepackage{tabularx}
\usepackage{ifthen}
\usepackage{eucal}
\usepackage{amssymb,amscd,amsthm, verbatim,amsmath,color,fancyhdr, mathrsfs}
\usepackage{graphicx}
\usepackage{turnstile}
\usepackage[colorinlistoftodos]{todonotes}

\newtheorem{thm}{Theorem}[section]
\newtheorem{prop}[thm]{Proposition}
\newtheorem{lem}[thm]{Lemma}

\newtheorem{cor}[thm]{Corollary}

\newtheorem{ex}[thm]{Example}
\newtheorem{defn}[thm]{Definition}

\begin{document}

\begin{center}
{\Large \textbf{An Inversion Statistic on the Hyperoctahedral Group}} \vspace*{0.5cm}
\end{center}

\vspace*{0.3cm}
\begin{center}
Hasan Arslan$^{a,1}$, Alnour Altoum$^{a,2}$,  Hilal Karakus Arslan$^{b,3}$ \\
$^{a}${\small {\textit{Department of  Mathematics, Faculty of Science, Erciyes University, 38039, Kayseri, Turkey}}}\\
$^{b}${\small {\textit{Hikmet Kozan Secondary School, Republic of Turkey Ministry of National Education, 38070, Kayseri, Turkey}}}\\
{\small {\textit{ $^{1}$hasanarslan@erciyes.edu.tr}}}~~{\small {\textit{$^{2}$ alnouraltoum178@gmail.com}}}\\
{\small {\textit{$^{3}$karakus545@hotmail.com }}}\\[0pt]
\end{center}

\begin{abstract}
In this paper, we introduce an inversion statistic on the hyperoctahedral group $B_n$ by using an  decomposition of a positive root system of this reflection group. Then we prove some combinatorial properties for the inversion statistic. We establish an enumeration system on the group $B_n$ and give an efficient method to uniquely derive any group element known its enumeration order with the help of the inversion table. In addition, we prove that the \textit{flag-major index} is equi-distributed with this inversion statistic on $B_n$. 
\end{abstract}


\\

\section{Introduction}
The main aim of this paper is to introduce a new inversion statistic, which can be considered as a partition of the length function on the hyperoctahedral group $B_n$. This inversion number is compatible with the length function on $B_n$, just as the inversion number on $S_n$. If a statistic is equi-distributed with the length function (i.e., the number of inversions) on a Coxeter group, then it is called as \textit{Mahonian}. We give a bijective proof of the equi-distribution of this inversion statistic and the \textit{flag-major index} on $B_n$ by defining a map $\phi: B_n \rightarrow B_n$  such that $$inv(w)=fmaj(\phi(w))$$ 
for each $w \in B_n$. We illustrate this fact for the group $B_3$. Inversion table of any element in the group $B_n$ shares the same properties with a number in the $B_n$-type number system. Therefore, we also supply an enumeration system on the hyperoctahedral group and provide an approach to uniquely obtain any group element, which is known its enumeration order, by means of its inversion table. Having written any positive integer in the $B_n$-type number system, it can be uniquely converted to an element of the group $B_n$ using the concept of inversion table. This provides convenience in algebraic and combinatorical studies related to group structure, as well as in studies in the field of cryptology.

\section{Preliminaries and Notation}

In this section, we recall the definition of a real reflection group $B_n$, which is also called a hyperoctahedral group. Throughout this paper, for any $m, n \in \mathbb{Z}$ such that $m \leq n$, we assume that $[m,n]:=\{m, m+1, \cdots, n\}$. Let $\mathbb{R}^n$ be an Euclidean space. Let $\{e_1,\cdots, e_n\}$ be the set of standard basis vectors of $\mathbb{R}^n$. In fact, a finite real reflection group $B_n\subset GL_n(\mathbb{R})$ is generated by the reflections $s_1,\cdots,s_{n-1}$ of order $2$ associated with the roots $e_2-e_1,\cdots, e_n-e_{n-1}$, respectively, and an exceptional reflection $t_1$ of order $2$ with root $e_1$. The set $S=\{t_1,s_1,\cdots, s_{n-1}\}$ is the canonical set of generators for the group $B_n$. It is well-known that $B_n$ is a semi-direct product of $S_{n}$ and $\mathcal{T}_{n}$, where $S_n$ is the symmetric group generated by $\{ s_1, \cdots, s_{n-1}\}$ and $\mathcal{T}_{n}$ is a reflection subgroup of $B_n$ generated by $\{ t_1, \cdots, t_n \}$, where $t_{i+1}:=s_{i}t_{i}s_{i}$ for each $1 \leq i \leq n-1$. Any element $w \in B_n$ can be uniquely written in the form
\[
w = \bigl(\begin{smallmatrix}
	1 & 2 &  ~~\cdots  & n \\
	(-1)^{r_1}\beta_1 & ~~(-1)^{r_2}\beta_2 & ~~\cdots   & ~~ (-1)^{r_n}\beta_n
\end{smallmatrix}\bigr)=\beta \prod_{k=1}^n t_{k}^{r_k}\in B_n,
\]
where $r_i \in \{0,1\}$ and we write $\beta = \bigl(\begin{smallmatrix}
	1 & 2 &  \cdots  & n \\
	\beta_1 & \beta_2 & \cdots   &  \beta_n
\end{smallmatrix}\bigr) \in S_n$ to mean that $ \beta_i= \beta(i)$ for all $i=1, \cdots, n$. If we take into account the group $B_n$ as a real reflection group with the following root system
$$
\Psi=\{\pm e_l,~~ \pm e_j\pm e_i ~~:~l \in [1,n],~1 \leq i \neq j \leq n\}.
$$
then we have the sets of positive and negative roots regarding with $\Psi$, which are, respectively, defined as follows:
\begin{equation*}
	\Psi^{+}=\{e_l, ~~e_j-e_i, ~~ e_j+e_i ~~:~~l \in [1,n],~~1 \leq i<j  \leq n\},
\end{equation*}
and $\Psi^{-}=-\Psi^{+}$. As is well-known from \cite{Humphreys}, $\Psi$ can be written as $\Psi=\Psi^{+} \bigsqcup \Psi^{-}$ a decomposition of $\Psi^{+}$ and $\Psi^{-}$. The length function $\textit{L}$ on $B_n$ associated with the root system $\Psi$ is defined as 
\begin{equation}\label{len}
	\textit{L}~:~B_n \rightarrow \mathbb{N},~~~\textit{L}(w)=\mid  w(\Psi^{+}) \cap \Psi^{-} \mid .
\end{equation}
Furthermore, the length $L(w)$ of $w$ is equal to the length of the minimal expression for $w$ in terms of the elements of $S$. The longest element of $B_n$ is $w_0=t_1 \cdots  t_n$ and it is also central. It is well-known that the longest element $w_0$ of the group $B_n$ can be expressed as a signed permutation in the following form:
\[
w_0=  \left(
\begin{matrix}
	1 & 2 & 3 & \cdots & n-1 & n \\
	-1 & -2 & -3 & \cdots &  -(n-1)  & -n
\end{matrix}
\right).
\]
We note here that the length of any reduced expression in $B_n$ takes value at most $n^2$, which is the length of $w_0$.

Let $\sigma=\sigma_1 \cdots \sigma_n$ be a one-line presentation of $\sigma= \bigl(\begin{smallmatrix}
    1 & 2 &  \cdots  & n \\
    \sigma_1 & \sigma_2 & \cdots   &  \sigma_n
  \end{smallmatrix}\bigr) \in S_n$. As is well-known from \cite{Stanley2011}, the inversion number, the descent set and the major index of $\sigma$ are respectively defined in the following way:
\begin{align*}
inv (\sigma)=&| \{(i,j) \in [1,n] \times [1,n] ~:~ i<j~and~\pi_i > \pi _{j}\} | \\
Des (\sigma)=&\{i \in [1,n-1] ~:~ \sigma_i > \sigma _{i+1}\} \\
maj (\sigma)=& \sum_{i \in Des (\sigma)}i.
\end{align*}
MacMahon proved in \cite{MacMahon1915} that the the number of inversions inv is equi-distributed with the major index maj over the symmetric group $S_n$, that is, 
$$\sum_{\sigma \in S_n}q^{inv(\sigma)}=\sum_{\sigma \in S_n}q^{maj(\sigma)}.$$

Following \cite{flag2001} we let, $\sigma_0:=t_1$ and for all $i \in [1,n-1]$,~$\sigma_i:=s_is_{i-1}\cdots s_1 t_1 \in B_n$. Thus, the collection $\{\sigma_0, \sigma_1, \cdots, \sigma_{n-1}\}$ is a different set of generators for $B_n$ and any $w\in B_n$ has a unique expression 
\begin{equation}\label{flag0}
w=\sigma_{n-1}^{k_{n-1}}\cdots \sigma_{2}^{k_{2}}\sigma_{1}^{k_{1}}\sigma_{0}^{k_{0}}
\end{equation}
with $0\leq k_i \leq 2i+1$ for all $0\leq i \leq n-1$. \textit{Flag-major index} was defined for the group $B_n$ as follows (see \cite{flag2001}): Let $w\in B_n$. Then
\begin{equation}\label{flag1}
fmaj(w)=\sum_{i=0}^{n-1}k_i.
\end{equation}
It is well-known from \cite{flag2001} that the flag-major index  is Mahonian, that is, 
\begin{equation}\label{flag2}
\sum_{w \in B_n}q^{fmaj(w)}=\prod_{i=1}^n [2i]_q=\prod_{i=1}^n q^{L(w)}
\end{equation}
where $q$ is an indeterminate and $[2i]_q=\frac{1-q^{2i}}{1-q}$ for every $i=1,\cdots, n$.

\begin{thm} [Adin-Roichman \cite{flag2001}]
Let $w \in B_n$. Then
\begin{equation}
fmaj(w)=2maj(w)+neg(w).
\end{equation}
where $neg(w)=|\{i\in [1,n]:w(i)<0\}|$ and maj is computed by using the following order on $\mathbb{Z}$: 
\[
-1<-2<\cdots<-n<1<2<\cdots < n
\]
\end{thm}
\begin{ex}
If $w=[2,-5,-3,-1,4]\in B_5$, then we have $maj(w)=1+2+3=6$, $neg(w)=3$ and so $fmaj(w)=15$, where the order we use when calculating the maj index is $-1<-2<-3<-4<-5<1<2<3<4<5$.
\end{ex}

In \cite{Raharinirina2020-0}, Raharinirina constructed a $B_n$-type number system and showed that any positive integer can be uniquely expressed in the $B_n$-type number system. 

\begin{defn}[\cite{Raharinirina2020-0}]
	The $B_n$-type number system is a radix base system that every positive integer $x$ can be uniquely expressed in the following form:
	\begin{equation}\label{def}
		x=\sum_{i=0}^{n-1} d_iB_i
	\end{equation}
	where $n \in \mathbb{Z}^+$, $d_i \in \{0, 1, 2, \cdots, 2i+1\}$ and $B_i=2^{i}i!$. 
\end{defn}
Then, for any positive integer $x$ in the $B_n$-type number system, we will use the notation 
\[
x=(d_{n-1}:d_{n-2}: \cdots : d_1: d_0).
\]

To write any positive integer $x$ in a $B_n$-type number system due to \cite{Raharinirina2020-0}, one proceeds with the following steps. In the first step, $x$ is divided by $2$ and the reminder $r_0$ is set to $d_0$ in the division process
\begin{equation*}
	x=2q_0+r_0.
\end{equation*}
Then divide $q_0$ by $4$ and the reminder $r_1$ sets to 
$d_1$ in the following division process
\begin{equation*}
	q_0=4q_1+r_1.
\end{equation*}
If we continue these operations by dividing $q_{i-1}$ by $2(i+1)$ and take $r_i=d_i$ in the expression
\begin{equation*}
	q_{i-1}=2(i+1)q_i+r_i
\end{equation*}
until the quotient $q_{n-1}$ is zero for some integer $n$. Thus, at the last step, we get 
\begin{equation*}
    q_{n-2}=2nq_{n-1}+r_{n-1}
\end{equation*}
and set $r_{n-1}$ as $d_{n-1}$. Eventually, the number $x$ is written in the form 
\begin{equation}\label{intrep}
	x=(d_{n-1}:d_{n-2}: \cdots : d_1: d_0)
\end{equation}
in the $B_n$-type base system.

Therefore, we have developed two algorithms that will facilitate the process of obtaining the representation in $B_n$-type number system of  any positive integer and vice versa. 
Any positive integer can be easily written in a $B_n$-type number system in a unique way when using the Python algorithm provided below:\\

\textbf{Algorithm 1:} \hspace*{\fill} \\
x=\textrm{int(input}('\textrm{Enter}~~\textrm{any}~~\textrm{positive}~~ \textrm{integer:}'))\\
m=2\\
\textrm{for}~~\textrm{i}~~\textrm{in}~~\textrm{range}(1,x):\\
$d=x \% m$\\
if $x > 0:$\\
$x=x//h$\\
m=m+2\\
else:\\
break\\
print(d, end=':') \\

The following example shows how this algorithm works:  

\begin{ex}
	Pick the integer $x=163$. The expression of $x$ in $B_{4}$-type number system is $x=(3:2:1:1)$.
\end{ex}

Alternatively, a positive integer can also be created from any number in the $B_n$-type number system with the help of the following Python algorithm:\\

\textbf{Algorithm 2:} \hspace*{\fill} \\
n=\textrm{int(input}('\textrm{Enter}~~\textrm{the}~~\textrm{index}~~ \textrm{of}~~\textrm{$B_n$}~~ \textrm{base}'))\\
f=1\\
x=0\\
\textrm{for}~~\textrm{i}~~\textrm{in}~~\textrm{range}(0,n):\\
d=\textrm{int(input}('\textrm{Enter}~~\textrm{a}~~\textrm{number}~~ \textrm{in}~~\textrm{$B_n$-type}~~ \textrm{number}~~\textrm{system}'))\\
if $i==0$ or $i==1:$\\
f=1\\
else:\\
$f=f*i$\\
$t=2**(i)*f$\\
z=d*t\\
x +=z\\
print('The decimal number is: ',x) 

\begin{ex}   
	Take $x=(10:12:3:9:10:5:1:1:0:1)$ as a number in $B_{10}$-type number system. It is actually an integer representation in the $B_{10}-type$ number system of the positive integer $x=1984199097$.
\end{ex}

\section{Revisited inversion statistic on group $B_n$}

In this section, we will give a various approach to the concept of the inversion statistic defined on the group $B_n$ than that of \cite{Raharinirina2020-1}. We will show that any element of the group $B_n$ can be uniquely represented in the $B_n$-type number system by means of the inversion statistic method. 

Now we define the set 
$$\Psi_i=\{e_{n+1-i},~~e_{n+1-i}-e_j,~~e_{n+1-i}+e_j~:~j<n+1-i \leq n\},$$ and $inv_i(w)=\mid w(\Psi_i)\cap \Psi^{-}\mid$
for each $i=1,\cdots, n$. The sequence $I(w) = (inv_1(w) : \cdots : inv_n(w))$ is called \textit{inversion table} of element $w \in B_n$. It should be noted here that an inversion table can be think of as a number in the $B_n$-type number system. If $x$ be the corresponding positive integer to the inversion table $I(w)$, then $x+1$ is said to be the \textit{rank} of $w$. This enables us to enumerate all the elements of the group $B_n$. The inversion table of $w$ can be essentially created by applying the rule given in the following theorem:

\begin{thm}\label{3}
For $w=\beta \prod_{k=1}^n t_{k}^{r_k} \in B_n$, we have
\begin{equation}\label{33}
inv_i(w)=r_{n+1-i}+2.\mid \{(j,n+1-i) : j<n+1-i, ~~\beta_j<\beta_{n+1-i}, r_{n+1-i}\neq 0\} \mid+inv_i(\beta)
\end{equation}
for all $i=1,\cdots,n$, where $inv_i(\beta)=\mid\{(j,n+1-i) : j<n+1-i,~~ \beta_j>\beta_{n+1-i} \} \mid$ in $S_n$ and $r_{n+1-i} \in \{0,~1\}$.  More precisely, $inv_i(w)=1+2.\mid \{(i,j) : i<j, ~~\beta_j<\beta_{n+1-i}\} \mid+inv_i(\beta)$ when $r_{n+1-i}=1$ and   $inv_i(w)=inv_i(\beta)$ when $r_{n+1-i}=0$. 
\end{thm}

\begin{proof}
Let $e_{n+1-i} \in \Psi_i$. Then $w(e_{n+1-i})=(-1)^{r_{n+1-i}}e_{\beta_{n+1-i}} \in\Phi^{-}$ if and only if $r_{n+1-i}=1$. Now let $e_{n+1-i}\pm e_j \in \Psi_i$. We denote $e_{n+1-i}\pm e_j$ by $e_{n+1-i}-(-1)^k e_j$, where $k$ is $0$ or $1$. Then we have $w(e_{n+1-i}\pm e_j)=(-1)^{r_{n+1-i}}e_{\beta_{n+1-i}}- (-1)^{k+r_j}e_{\beta_j}$, which lies in  $\Psi^{-}$ if and only if either $r_{n+1-i}=1$ and $\beta_j< \beta_{n+1-i}$ (where $k$ takes exactly two values) or $k+r_j=2$ and $\beta_j> \beta_{n+1-i}$. This completes the proof. Clearly, $inv_i(w)=inv_i(\beta)$ if $r_{n+1-i}=0$ . This completes the proof.
\end{proof}

The formula given in (\ref{33}) is nothing else but a special case of the formula expressed in Theorem 4.5 in \cite{hasan2022} (the case $m=2$). For all $i=1,\cdots,n$, we get $inv_i(w) \in [0,2(n-i)+1]$.

Let $inv(w)$ denote the sum of $i$-inversions of the permutation $w \in B_n$. It is clear that $L(w)=inv(w)$. Subsequently, one can practically determine the length of $w$ with the help of its inversion table. 

 \begin{ex}
Let $w= \bigl(\begin{smallmatrix}
    1 & ~~2 & ~~ 3&~~ 4&~~5&~~6 &~~7&~~8\\
   2&~~ 4 &~~ 1 &~~-3&~~ 6&~~ 7&~~-5&~~8
\end{smallmatrix}\bigr)\in B_8.$ Taking into account the equation (\ref{33}) we obtain the inversion table of $w$ as $I(w)=(0:11:0:0:6:2:0:0)$, and so we conclude that the length of $w$ is $L(w)=19$ and that the rank of $w$ is $507185$ using Algorithm 2. On the other hand, the reduced expression of $w$ is $s_1s_2s_3s_1t_1s_1s_2s_3s_4s_3s_2s_1t_1s_1s_2s_3s_4s_5s_6$ according to the canonical generating set $S=\{t_1,s_1,\cdots,s_7\}$, and so $L(w)=19$ from another viewpoint.
 \end{ex}
Taking into account the equation (\ref{33}), we can give the following result for the longest element $w_0$ in $B_n$.

\begin{cor}\label{longest}
	Let $w_0$ be the longest element of the group $B_n$. Then the inversion table of $w_0$ is $$I(w_0)=(d_{n-1}:d_{n}: \cdots : d_1: d_0)=(2n-1:2n-3: \cdots:5:3:1).$$
	Therefore, it is clear that the order of group $B_n$ is 
	$$\mid  B_n \mid=\prod_{i=0}^{n-1}(d_i+1)=2^{n}n!.$$
\end{cor}

It is well-known from \cite{Humphreys} that the exponents of the group $B_n$ are $1,~3,\cdots, 2n-1$, respectively. Note that all components in the inversion table of $w_0$ in Corollary \ref{longest} exactly coincide with the exponents of the group $B_n$.

Now, on the contrary, we set up a fruitful technique of how to create the signed permutation $w\in B_n$ from a given rank value in the following way:
First of all, consider a rank  $k$. Subsequently, turn $k-1$ into a number in $B_n$-type number system. Let us show $k-1$ by $(d_{n-1}:\cdots :d_1:d_0)\in B_n$. Essentially, we want to obtain a signed permutation $w$ such that
\[
inv_i(w)=d_{n-i}~\textrm{for}~\textrm{all}~i=1,\cdots,n.
\]

We will build a signed permutation $w=\bigl(\begin{smallmatrix}
    1 & 2 &  \cdots &n-1 & n \\
    w_1 & w_2 & \cdots &w_{n-1}  &  w_n
  \end{smallmatrix}\bigr)$ associated with $(d_{n-1}:\cdots :d_1:d_0)$ by proceeding the following steps:
\begin{itemize}
     \item Having listed all possible values that the desired signed permutation can take in the following order
\begin{equation}\label{converse1}
n > \cdots >1>-1>\cdots>-n
\end{equation}
enumerate them up from $0$ to $2n-1$ by starting with the leftmost value in (\ref{converse1}). Then find the value corresponding to the number $d_{n-1}$ from (\ref{converse1}) and set it as $w_n$.\\
 \item Say $w_n:=(-1)^{r_i}i$, where $r_i \in \{0,1\}$. Extract the terms $i,- i$  from (\ref{converse1}). After that, reorder the remaining values as 
\begin{equation}\label{converse2}
\begin{split}
n>\cdots>i+1>i-1>\cdots>1> &-1>\cdots>- (i-1)> -(i+1)> \cdots>-n
\end{split}
\end{equation}
and renumber them from $0$ to $2(n-1)-1$ by starting with the leftmost term. Then determine the value corresponding to the number $d_{n-2}$ from (\ref{converse2}) and assign it to $w_{n-1}$.
     \item Carry out the same procedure for  (\ref{converse2}) and determine in this manner $w_{n-2}$.
\item Proceed these iterations until you determine all $w_i$ values for each $1 \leq i \leq n$.
 \end{itemize}

Let us consider the following example to make this method clear.

\begin{ex} 
We will find the $1464993^{th}$ group element of $B_8$. If we apply Algorithm 1 to the positive integer $1464992$, then we get the inversion table of this element which we are looking for as $I(w)=(2:3:9:5:0:4:0:0)$.

\scalebox{0.7}{
	\begin{tabular}{|m{0.9cm}|m{0.9cm}|m{.4cm}m{.4cm}|m{.4cm}m{.4cm}|m{.4cm}m{.4cm}|
 m{.4cm}m{.4cm}|m{.4cm}m{.4cm}|m{.4cm}m{.4cm}|m{.4cm}m{.4cm}|m{.4cm}m{.4cm}|}
		\hline
		\multirow{2}{*}{Step 1}&P.P.V.&8& 7 &\color{blue}6& 5& 4& 3&2 &1& -1 &-2&-3 &-4& -5& -6&-7 &-8\\
		\cline{2-18}
		&P.I.V.&0& 1 &\color{blue}2 &3 &4& 5& 6 &7 &8& 9&10 &11&12& 13&14& 15\\
		\hline
	\end{tabular}}
 
 \scalebox{0.7}{
	\begin{tabular}{|m{0.9cm}|m{0.9cm}|m{.4cm}m{.4cm}|m{.4cm}m{.4cm}|m{.4cm}m{.4cm}|
 m{.4cm}m{.4cm}|m{.4cm}m{.4cm}|m{.4cm}m{.4cm}|m{.4cm}m{.4cm}|}
		
		\multirow{2}{*}{Step 2}&P.P.V.&8& 7 & 5& \color{blue}4& 3&2 &1& -1 &-2&-3 &-4& -5& -7 &-8\\
		\cline{2-16}
		&P.I.V.&0& 1 &2 &\color{blue}3 &4& 5& 6 &7 &8& 9&10 &11&12& 13\\
		\hline
	\end{tabular}}

 \scalebox{0.7}{
	\begin{tabular}{|m{0.9cm}|m{0.9cm}|m{.4cm}m{.4cm}|m{.4cm}m{.4cm}|m{.4cm}m{.4cm}|
 m{.4cm}m{.4cm}|m{.4cm}m{.4cm}|m{.4cm}m{.4cm}|}
		
		\multirow{2}{*}{Step 3}&P.P.V.&8& 7 & 5&  3&2 &1& -1 &-2&-3 &\color{blue} -5& -7 &-8\\
		\cline{2-14}
		&P.I.V.&0& 1 &2 &3 &4& 5& 6 &7 &8& \color{blue}9&10 &11\\
		\hline
	\end{tabular}}

\scalebox{0.7}{
	\begin{tabular}{|m{0.9cm}|m{0.9cm}|m{.4cm}m{.4cm}|m{.4cm}m{.4cm}|m{.4cm}m{.4cm}|
 m{.4cm}m{.4cm}|m{.4cm}m{.4cm}|}
		
		\multirow{2}{*}{Step 4}&P.P.V.&8& 7 &  3&2 &1&\color{blue} -1 &-2&-3 &-7 &-8\\
		\cline{2-12}
		&P.I.V.&0& 1 &2 &3 &4& \color{blue}5& 6 &7 &8&9\\
		\hline
	\end{tabular}}

 \scalebox{0.7}{
	\begin{tabular}{|m{0.9cm}|m{0.9cm}|m{.4cm}m{.4cm}|m{.4cm}m{.4cm}|m{.4cm}m{.4cm}|
 m{.4cm}m{.4cm}|}
		
		\multirow{2}{*}{Step 5}&P.P.V.&\color{blue}8& 7 &  3&2 &-2&-3 &-7 &-8\\
		\cline{2-10}
		&P.I.V.&\color{blue}0& 1 &2 &3 &4& 5& 6 &7 \\
		\hline
	\end{tabular}}

  \scalebox{0.7}{
	\begin{tabular}{|m{0.9cm}|m{0.9cm}|m{.4cm}m{.4cm}|m{.4cm}m{.4cm}|m{.4cm}m{.4cm}|}
		
		\multirow{2}{*}{Step 6}&P.P.V.& 7 &  3&2 &-2&\color{blue}-3 &-7\\
		\cline{2-8}
		&P.I.V.&0& 1 &2 &3 &\color{blue}4& 5 \\
		\hline
	\end{tabular}}

   \scalebox{0.7}{
	\begin{tabular}{|m{0.9cm}|m{0.9cm}|m{.4cm}m{.4cm}|m{.4cm}m{.4cm}|}
		
		\multirow{2}{*}{Step 7}&P.P.V.&\color{blue} 7 & 2 &-2& -7\\
		\cline{2-6}
		&P.I.V.&\color{blue}0& 1 &2 &3\\
		\hline
	\end{tabular}}

  \scalebox{0.7}{
	\begin{tabular}{|m{0.9cm}|m{0.9cm}|m{.4cm}m{.4cm}|}
		
		\multirow{2}{*}{Step 8}&P.P.V.& \color{blue}2 &-2\\
		\cline{2-4}
		&P.I.V.&\color{blue}0& 1 \\
		\hline
	\end{tabular}}\\

In the above table, the abbreviations of possible values that the desired signed permutation can take and possible inversion values are represented by P.P.V. and P.I.V., respectively. Considering the above table, then $w\in B_8$ is built up as 
\begin{equation*}
 w = \bigl(\begin{smallmatrix}
    1 & 2 & ~3& 4& ~5 & ~6 & 7 & 8\\
    2 & 7 & -3& 8& -1 & -5 & 4 & 6
  \end{smallmatrix}\bigr).
\end{equation*}
\end{ex}

Based on the above facts, we may state the following result without proof.

\begin{prop}
Let
\begin{align*}
\mathcal{T}_{2,n}&=\{(a_1 : \cdots : a_n) ~:~ 0 \leq a_i \leq 2(n-i+1)-1,~i=1,\cdots, n\}\\
&=[0,2n-1]\times [0,2n-3]\times \cdots \times [0,3]\times [0,1]. 
\end{align*}
The map $I: B_n \rightarrow \mathcal{T}_{2,n}$ that assigns each permutation to its inversion table is a bijection. 
\end{prop}
Therefore, we conclude that the inversion table $I(w)$ is basically another way to represent a permutation $w\in B_n$. The fact that  $inv(w)=\sum_{i=1}^n inv_i (w)$ for any $w \in B_n$ allows us to give a new approach to the proof of Poincar\'e polynomial for $B_n$ in the sense of $S_n$ (see \cite{Stanley2011}).

\begin{thm} \label{stanley}
The Poincar\'e polynomial for $B_n$ is in the following form:
\begin{equation*}
\sum_{w \in B_n}q^{inv(w)}=\prod_{i=1}^n [2i]_q
\end{equation*}
where $q$ is an indeterminate and $[2i]_q=\frac{1-q^{2i}}{1-q}$ for every $i=1,\cdots, n$.
\end{thm}

\begin{proof}
If $I(w)=(inv_1(w):\cdots: inv_n(w))=(a_1:\cdots: a_n)$ then $inv(w)=\sum_{i=1}^na_i.$ Hence 
\begin{align*}
\sum_{w \in B_n}q^{\textit{L}(w)}=\sum_{w \in B_n}q^{inv(w)}=&\sum_{a_1=0}^{2n-1}~~\sum_{a_2=0}^{(2n-3)}\cdots \sum_{a_n=0}^{1}q^{a_1+a_2+\cdots+a_n}\\
=&(\sum_{a_1=0}^{2n-1}q^{a_1}) (\sum_{a_2=0}^{2n-3}q^{a_2}) \cdots (\sum_{a_n=0}^{1}q^{a_n})\\
=&\prod_{i=1}^n [2i]_q
\end{align*}
as desired.
\end{proof}

Let $\pi=[\pi_1,\cdots,\pi_{n-1}] \in B_{n-1}$. We want to observe how the insertion of $n$ (resp. $-n$) into the permutation $\pi$ affects the inversion statistic. There are clearly $n$ places where we can put $n$ (resp. $-n$) into the permutation $[\pi_1,\cdots, \pi_{n-1}]$. More precisely, for each $i=1,\cdots,n-1$ there is one place immediately after $\pi_i$ which is called space $i$ and there is one more place immediately before $\pi_1$ which we call space $0$. It is easy to see that the following insertion lemma holds. We denote by $\pi_{n,i}$ (resp. $\pi_{-n,i}$) the permutation in $B_n$ by inserting $n$ (resp. $-n$) into place $i$ in $\pi$.

\begin{lem} \label{insertion}
Suppose that $\pi=[\pi_1,\cdots,\pi_{n-1}]$ is a permutation in $B_{n-1}$. Then we have
\begin{enumerate}
    \item $inv\pi_{n,i}=n-i-1+inv\pi$
    \item $inv\pi_{-n,i}=n+i+inv\pi.$
\end{enumerate}
\end{lem}

\begin{ex}
We consider $\pi=[-3,1,2,-4,-5] \in B_5$. Then the inversion table of $\pi$ is $I(\pi)=(9:7:1:1:1)$and $inv\pi=19$. If $\pi_{6,2}=[-3,1,6,2,-4,-5]$, then $I(\pi_{6,2})=(10:8:2:0:1:1)$ and $inv\pi_{6,2}=22$. If $\pi_{-6,2}=[-3,1,-6,2,-4,-5]$, then $I(\pi_{-6,2})=(10:8:2:5:1:1)$ and $inv\pi_{-6,2}=27$.
\end{ex}
We immediately give the next corollary as a result of Lemma \ref{insertion}.

\begin{cor}
Let $\pi=[\pi_1,\cdots,\pi_{n-1}] \in B_{n-1}$. Then we have
\begin{enumerate}\label{suminsertions1}
       \item $\sum_{i=0}^{n-1}q^{inv\pi_{n,i}}=[n]_q q^{inv\pi}$,
       \item $\sum_{i=0}^{n-1}q^{inv\pi_{-n,i}}=q^n [n]_q q^{inv\pi}$.
\end{enumerate} 
\end{cor}
Hence for any $\pi \in B_{n-1}$, we conclude that
\begin{equation}\label{suminsertions2}
 \sum_{i=0}^{n-1}(q^{inv\pi_{n,i}}+q^{inv\pi_{-n,i}})=([n]_q +q^n [n]_q) q^{inv\pi}=[2n]_q q^{inv\pi}.  
\end{equation}
Hence, it is not hard to prove by induction that
\begin{equation}
\sum_{\pi \in B_n}q^{inv\pi}=[2n]_q \sum_{\tau \in B_{n-1}}q^{inv\tau}=\prod_{i=1}^n [2i]_q.
\end{equation}

For any $w\in B_n$ with the flag major index $fmaj(w)=k_0+k_1+\cdots+k_{n-1}$, where $0\leq k_i \leq 2i+1$,~$i=0,1,\cdots,n-1$. Then we have
\begin{align*}
\sum_{w \in B_n}q^{fmaj(w)}=&\sum_{k_{n-1}=0}^{2n-1}~~\sum_{k_{n-2}=0}^{2n-3}\cdots \sum_{k_0=0}^{1}q^{k_0+k_1+\cdots+k_{n-1}}\\
=&(\sum_{k_{n-1}=0}^{2n-1}q^{k_{n-1}}) (\sum_{k_{n-2}=0}^{2n-3}q^{k_{n-2}}) \cdots (\sum_{k_0=0}^{1}q^{k_0})\\
=&\prod_{i=1}^n [2i]_q.
\end{align*}

Considering Theorem \ref{stanley} and the above result together, we conclude that the inversion statistic is equi-distrubuted with and the flag-major index over $B_n$. Moreover, we can define a map $\phi: B_n \rightarrow B_n$  such that $inv(w)=fmaj(\phi(w))$ for each $w \in B_n$ as follows: Let the inversion table of $w$ be $I(w)=(a_{n-1}:\cdots:a_1:a_0)$. If we define
$$\phi(w)=\sigma_{n-1}^{a_{n-1}}\cdots \sigma_{1}^{a_{1}}\sigma_{0}^{a_{0}}$$
then it is obvious that $\phi$ is a bijection and $inv(w)=\sum_{i=0}^{n-1} a_i=fmaj(\phi(w)),$  where $\sigma_{i}$ is defined as in (\ref{flag0}) for each $i,~0 \leq i \leq n-1$.

\begin{ex}

In Table \ref{tab:table1}, we respectively record the ranks and the inversion tables of the forty-eight elements of $B_3$ and their images under the map $\phi$. Note that $inv(w)=fmaj(\phi(w))$ holds for each $w \in B_3$. In the following table, we will denote any permutation $w$ in $B_3$ by $w_1w_2w_3$.
\begin{table}[h!]
\centering	
\quad \caption{Inversion table of the group $B_3$.}
	\label{tab:table1}

 \scalebox{0.8}{
 
	\begin{tabular}{|c|c|c|c||c|c|c|c||c|c|c|c|}
		\hline
		Rank&\textbf{$w_1 w_2 w_3$} & \textbf{$I(w)$} & \textbf{$\phi(w)$}&Rank&\textbf{$w_1 w_2 w_3$}  & \textbf{$I(w)$} &\textbf{$\phi(w)$}& Rank&\textbf{$w_1 w_2 w_3$} & \textbf{$I(w)$}&\textbf{$\phi(w)$}\\
		\hline\hline
		1&1 2 3 & 0:0:0 &1 2 3&17&2 3 1 & 2:0:0 &-2 -3 1&33 &1 3 -2 & 4:0:0&3 -1 -2\\
		\hline
		2&-1 2 3 & 0:0:1 &-1 2 3&18& -2 3 1 & 2:0:1 &2 -3 1&34& -1 3 -2 & 4:0:1&-3 -1 -2\\
		\hline
		3&2 1 3 & 0:1:0 &-2 1 3&19& 3 2 1 & 2:1:0 &3 -2 1&35& 3 1 -2 & 4:1:0&1 3 -2\\
		\hline
		4&-2 1 3 & 0:1:1 &2 1 3 &20& -3 2 1 & 2:1:1 &-3 -2 1&36 &-3 1 -2 & 4:1:1&-1 3 -2\\
		\hline
		5&2 -1 3 & 0:2:0 &-1 -2 3&21& 3 -2 1 & 2:2:0 &2 3 1&37& 3 -1 -2 & 4:2:0&-3 1 -2\\
		\hline
		6&-2 -1 3 & 0:2:1 &1 -2 3&22&-3 -2 1 & 2:2:1 &-2 3 1&38 &-3 -1 -2 & 4:2:1&3 1 -2\\
		\hline
		7&1 -2 3 & 0:3:0 &2 -1 3&23& 2 -3 1 & 2:3:0 &-3 2 1&39& 1 -3 -2 & 4:3:0&-1 -3 -2\\
		\hline
		8&-1 -2 3 & 0:3:1 &-2 -1 3 &24 &-2 -3 1 & 2:3:1 &3 2 1&40& -1 -3 -2 & 4:3:1&1 -3 -2\\
		\hline
		9&1 3 2 & 1:0:0 &-3 1 2&25 &2 3 -1 & 3:0:0 &-1 -2 -3&41& 1 2 -3 & 5:0:0&2 3 -1\\
		\hline
		10&-1 3 2 & 1:0:1 &3 1 2&26& -2 3 -1 & 3:0:1 &1 -2 -3&42 &-1 2 -3 & 5:0:1&-2 3 -1\\
		\hline
		11&3 1 2 & 1:1:0 &-1 -3 2&27& 3 2 -1 & 3:1:0 &2 -1 -3&43& 2 1 -3 & 5:1:0&-3 2 -1\\
		\hline
		12&-3 1 2 & 1:1:1 &1 -3 2&28& -3 2 -1 & 3:1:1 &-2 -1 -3&44& -2 1 -3 & 5:1:1&3 2 -1\\
		\hline
		13&3 -1 2 & 1:2:0 &3 -1 2&29& 3 -2 -1 & 3:2:0 &1 2 -3&45& 2 -1 -3 & 5:2:0&-2 -3 -1\\
		\hline
            14&-3 -1 2 & 1:2:1 &-3 -1 2&30& -3 -2 -1 & 3:2:1 &-1 2 -3&46& -2 -1 -3 & 5:2:1&2 -3 -1\\
		\hline
		15&1 -3 2 & 1:3:0 &1 3 2&31 &2 -3 -1 & 3:3:0 &-2 1 -3&47& 1 -2 -3 & 5:3:0&3 -2 -1\\
		\hline
		16&-1 -3 2 & 1:3:1 &-1 3 2&32& -2 -3 -1 & 3:3:1 &2 1 -3&48 &-1 -2 -3 & 5:3:1&-3 -2 -1\\
		\hline
	\end{tabular}}
\end{table}
\end{ex}

\textbf{Question:}
It is a natural question to ask here when considering both the inversion number and the flag-major index together, then how exactly can Haglund-Remmel-Wilson identity be defined for the group $B_n$? As a matter of fact, the flag major part of Haglund-Remmel-Wilson identity for the group $B_n$ was given in \cite{ding2023identities} by depending on $q$-Stirling numbers of the second kind in type B. In the case of symmetric group, the proof of Haglund-Remmel-Wilson identity, 
$$\sum_{\sigma \in S_n}q^{inv(\sigma)}\prod_{j \in Des(\sigma)}(1+\frac{z}{q^{1+inv_j(\sigma)}})=\sum_{\sigma \in S_n}q^{maj(\sigma)}\prod_{j=1}^{des(\sigma)} (1+\frac{z}{q^j})$$
where $inv_j(\sigma)=| \{(i,j) \in [1,n] \times [1,n] ~:~ i<j~and~\sigma_i > \sigma _{j}\} |$ is defined just as Theorem \ref{3}, was proved by Remmel and Wilson in \cite{remmel2015}.

\end{document}